\documentclass[a4paper,11pt]{amsart}
\usepackage{amssymb,mathrsfs}
\usepackage{amsthm}
\usepackage{amsmath}
\usepackage{latexsym}
\usepackage{fancyhdr}
\usepackage{ascmac}
\usepackage{color}
\usepackage[all]{xy}
\usepackage{amscd}

\theoremstyle{plain}
\newtheorem{thm}{Theorem}[section]
\newtheorem{prop}[thm]{Proposition}
\newtheorem{cor}[thm]{Corollary}
\newtheorem{lem}[thm]{Lemma}

\newtheorem{rmk}[thm]{Remark}

\makeatletter
\@namedef{subjclassname@2020}{\textup{2020} Mathematics Subject Classification}
\makeatother

\makeatletter
 
  \@addtoreset{equation}{section}
\makeatother

\newcommand{\bQ}{\overline{\mathbb{Q}}}

\newcommand{\bZp}{\overline{\mathbb{Z}}_p}

\newcommand{\C}{\mathbb{C}}
\newcommand{\R}{\mathbb{R}}
\newcommand{\Q}{\mathbb{Q}}

\newcommand{\Z}{\mathbb{Z}}

\newcommand{\lra}{\longrightarrow}

\newcommand{\E}{\mathcal{E}}

\newcommand{\A}{\mathbb{A}}
\newcommand{\D}{\mathcal{D}}
\newcommand{\U}{\mathcal{U}}

\renewcommand{\O}{\mathcal{O}}

\newcommand{\ds}{\displaystyle}

\newcommand{\G}{\Gamma}

\newcommand{\uk}{\underline{k}}

\newcommand{\bs}{\backslash}

\title[On finiteness theorems for automorphic forms]
{On finiteness theorems for automorphic forms}
\author{Takuya Yamauchi}
\thanks{A part of this work is done during the author was partially supported by 
 JSPS KAKENHI Grant Number (B) No.19H01778.}


\address{Takuya Yamauchi \\ 
Mathematical Inst. Tohoku Univ.\\
 6-3,Aoba, Aramaki, Aoba-Ku, Sendai 980-8578, JAPAN}
\email{yamauchi@math.tohoku.ac.jp or takuya.yamauchi.c3@tohoku.ac.jp}

\begin{document} 
\begin{abstract} 
In this paper, for any Shimura datum $(G,\mathcal D)$ satisfying reasonable conditions 
that many interesting cases satisfy, we prove 
some finiteness theorems for any graded vector space consisting of  
automorphic forms on $\mathcal{D}$ of some weights over the graded ring of automorphic forms on $X$ with 
positive parallel weights. 
We also discuss the integral base ring which we can work on. 
To realize automorphic forms as global sections on some coherent sheaves on the minimal compactification,  
we use the notion of reflexive sheaves and higher Koecher principle due to 
Kai-Wen Lan. 
Further, we give a slightly modified version of finiteness results for Siegel modular forms by using only the results of 
Chai-Faltings.         
\end{abstract}

\maketitle 
\section{Introduction}\label{intro} We refer \cite{Lan-example}, \cite{Lan-van} and \cite{Milne-intro} for Shimura data and Shimura varieties. 
Let $\mathcal D$ be the Hermitian symmetric domain associated to a Shimura datum $(G,\mathcal D)$ where 
$G$ is a connected reductive group over $\Q$. 
Let $G(\R)^+$ be the connected component of $G(\R)$ with the identity in 
the real topology. Put $G(\Q)^+=G(\Q)\cap G(\R)^+$. 
Let $(G^{\rm ad},\mathcal D^+)$ be the connected 
Shimura datum for $(G,\mathcal D)$ such that $\mathcal D^+$ is a connected component of $\mathcal D$ and 
$G(\R)^+$ acts transitively on $\mathcal D^+$. 
Let $\A_f=\widehat{\Z}\otimes_\Z\Q$ be the finite part of the ring of adeles of $\Q$.  
For any open compact subgroup $\mathcal U$ of $G(\A_f)$, put 
\begin{equation}\label{xu}
X_{\mathcal U}:=G(\Q)\bs \mathcal D\times G(\A_f)/\mathcal U\simeq 
G(\Q)^+\bs \mathcal D^+\times G(\A_f)/\mathcal U =\coprod_{i\in I}\G_i\bs \mathcal D^+ 
\end{equation}
where $G(\A_f)=\coprod_{i\in I}G(\Q)^+g_i \mathcal U$ and $\G_i=(g_i \mathcal U g^{-1}_i)\cap G(\Q)^+$. 
Clearly each $\G_i$ is commensurable with $G(\Z)$. Here $G(\Z)$ is defined by, first, choosing    
an embedding $\iota_N:G\hookrightarrow GL_N $ for some positive integer $N$ and then by talking the pullback of 
$GL_N(\Z)$ under $\iota_N$. 
It is well-known (cf. Theorem 1 of \cite{Satake-some}) that  
each component $\G_i\bs \mathcal D^+$ has a structure as a quasi-projective algebraic variety over $\C$ even when $\G_i$ is not neat. 
Let $K$ be a maximal compact subgroup of $G(\R)$ and $K_{\C}$ be 
its complexification. Then $\mathcal{D}=G(\R)/Z_G(\R)K$. 
For $\gamma\in G(\R)^+$ and $z\in \mathcal{D}^+$ we write $\gamma \cdot z$ for the natural left action. 
For each algebraic, finite dimensional representation $\rho$ of $K_\C$ with the representation space $V_\rho$, 
we define the holomorphic vector bundle $(G^+(\R) \times_{K,\rho|_K} V_\rho(\C))/K$  on $\D^+$ as 
a quotient of $G^+(\R)\times V_\rho(\C)$ by  
the relation $(g,v)\sim (gk,\rho^{-1}(k)v)$ for $(g,v)\in G^+(\R)\times V_\rho(\C)$ and $k\in K$. 
Since $\D^+$ is simply connected, the above holomorphic vector bundle is trivialized. 
Therefore, there is 
 a canonical automorphic factor associated to $\rho$:  
\begin{equation}\label{auto-factor}
J_\rho:G(\R)^+ \times \mathcal{D}^+\lra {\rm Aut}_\C(V_\rho)
\end{equation}
which is holomorphic in the complex variables of $\mathcal{D}^+$ and it satisfies the cocycle condition. 
We can also associate the automorphic vector bundle on $X_\U$ by 
$$W_{\rho,\mathcal{U}}=Z_G(\R)G(\Q)^+\bs (G(\R)\times G(\A_f)\times V_\rho(\C))/(K\times \mathcal{U})
\simeq \coprod_{i\in I}\G_i\bs (\mathcal D^+ \times V_\rho(\C))  $$
with the relation $(g_\infty,g_f,v)\sim (z_\infty \gamma g_\infty k ,\gamma g_fu,\rho^{-1}(k)v),\ 
z_\infty\in Z_G(\R), \gamma\in G(\Q)^+,\ k\in K ,\ u\in \mathcal{U}$ for 
$(g_\infty,g_f,v)\in \mathcal{D}\times G(\A_f)\times V_\rho$ and each $\G_i$ acts on 
$\mathcal D^+ \times V_\rho(\C)$ by 
$\gamma_i(Z,v)=(\gamma_i Z,J(\gamma_i,Z)v)$. 

Fix an algebraic (or holomorphic) character $\lambda:K_\C\lra \C^\times$ which is a positive parallel weight \cite{Lan-van} 
(equivalently, it is also said to be positive of rational type in the classical language \cite{Satake-some}). 
 We will specify $\lambda$ when we apply the results to Siegel modular forms.   
For any arithmetic subgroup $\G\subset G(\Q)^+$ and an algebraic representation $\rho$ of $K_C$ as above, we define the space $M_\rho(\G)$ which 
consisting of all holomorphic $V_\rho(\C)$-valued functions $F:\mathcal{D}^+\lra V_\rho(\C)$ enjoying the conditions:
\begin{itemize}
\item $F(\gamma\cdot z)=J_\rho(\gamma,z)F(z)$ for any $\gamma\in \G$,  
\item $\ds\lim_{z\to\partial \mathcal{D}}J_\rho(\delta,z)^{-1}F(\delta\cdot z)$ is finite for any $\delta\in G(\Q)^+$
\end{itemize}
where $\partial \mathcal{D}$ is the boundary of Satake compactification \cite{Satake-cpt}  or 
Baily-Borel compactification \cite{BB}. 
In this paper we call $F$ a (classical) automorphic form of weight $\rho$ with respect to $\G$. 
If we replace the above second condition with 
\begin{itemize}
\item $\ds\lim_{z\to\partial \D}J_\rho(\gamma,z)^{-1}F(\gamma\cdot z)=0$ for any $\gamma\in \mathcal{G}(\Q)^+$,
\end{itemize}
then we call $F$ a (classical) cusp form of weight $\rho$ with respect to $\G$. 
We denote by $S_\rho(\G)$ the space of all cusp forms of weight $\rho$ with respect to $\G$. It is well-known 
that both of 
$M_\rho(\G)$ and $S_\rho(\G)$ are finite dimensional vector spaces over $\C$. 
We also  define the graded vector spaces by 
\begin{equation}\label{def1}
M_{\rho,\lambda,\ast}(\G):=\bigoplus_{k\in \Z}M_{\rho\otimes\lambda^k}(\G),\ 
S_{\rho,\lambda,\ast}:=\bigoplus_{k\in \Z}S_{\rho\otimes\lambda^k}(\G).
\end{equation}
 
Put 
\begin{equation}\label{def2}
M_{\lambda,\ast}(\G)=M_{\textbf{1},\ast}(\G).
%,\ S_{\lambda,\ast}(\G)=S_{\textbf{1},\ast}(\G)
\end{equation}
where $\textbf{1}$ stands for the trivial representation of $K_\C$. 
Put
\begin{equation}\label{def3}
M_{\rho,\lambda,\ast}(\U)=\bigoplus_{i\in I}M_{\rho,\lambda,\ast}(\G_i),\ 
S_{\rho,\lambda,\ast}(\U)=\bigoplus_{i\in I}S_{\rho,\lambda,\ast}(\G_i),\ 
M_{\lambda,\ast}(\U)=\bigoplus_{i\in I}M_{\textbf{1},\lambda,\ast}(\G_i)
\end{equation}
with respect to (\ref{xu}). 
It will be revealed in the course of proofs of the main theorems that (\ref{def3}) can be defined as the global sections of 
coherent sheaves related to $W_{\rho,\U}$. 

The graded ring $M_{\lambda,\ast}(\G_i)$ (or  $M_{\lambda,\ast}(\U)$)  is consisting of   
automorphic forms of scalar weights proportional to $\lambda$. 
For such a $\lambda$,  
$M_{\lambda,\ast}(\G_i)$ (or  $M_{\lambda,\ast}(\U)$)  is finitely 
generated by Theorem \ref{cla} below. 
For positive parallel weights,  
all cases are described in Section 3.3 of \cite{Lan-van}. 
For example, when $G$ is ${\rm Res}_{F/\Q}SL_2/F$ or ${\rm Res}_{F/\Q}GL_2/F$ for any totally real field $F$ of degree $g$, 
then the weights of characters are parametrized by $g$-tuple non-negative integers 
$(k_1,\ldots,k_g)$ and they are said to be positive and of rational type if 
$k_1=\ldots=k_g>0 $. We have the same condition for the symplectic 
group $Sp_{2g}$ over $\Q$ of rank $g$ whose corresponding highest weights of characters are 
parametrized by $g$-tuple integers $(k_1,\ldots,k_g)$. When $G=GSp_{2g}/\Q$, the positive pararell weights are 
given by the same weights for $Sp_{2g}/\Q$ by ignoring the similitude part (see Section \ref{app}).    

Henceforth we assume 
\begin{equation}\label{G3}
\begin{array}{l}
\mbox{{\rm (dim)} Every $\Q$-simple factor of each component of $X_\G$ or $X_{\U}$ is compact or of dimension}\\
\mbox{\hspace{9mm}  greater than one}.
\end{array}
\end{equation}
Many interesting cases including Hilbert modular varieties, Siegel modular varieties, 
and unitary Shimura varieties satisfy the above condition, and it plays an important role in applying Serre's extension theorem (see the discussion in lines between (\ref{hkp}) and  (\ref{mini-auto})).   

The following theorem seems to be well-known for some cases for experts 
(and even to some non-experts) after Cartan Seminaires (however, this is not a paper directed to experts on Shimura varieties but to the community working on classical modualr forms 
and even over $\C$, most people do not know a reference in the vector-valued case
with my experience from discussions at conferences).
\begin{thm}\label{cla}
Assume (\ref{G3}) for $G$ and an arithmetic subgroup $\G$ of $G(\Q)^+$.  
For any algebraic representation $\rho$ of $K_\C$ and any positive algebraic character of $K_\C$ which is of rational type, 
the graded vector spaces $M_{\rho,\lambda,\ast}(\G)$ and $S_{\rho,\lambda,\ast}(\G)$ are finitely generated over 
the graded ring $M_{\lambda,\ast}(\G)$. 
\end{thm}
Note that $\G$ in the claim needs not to be congruent. 
A key is to realize automorphic forms and cusp forms as global sections of coherent sheaves on the minimal 
compactification of $X_\G$. However the minimal compactification is highly singular in general and therefore 
it seems difficult to directly construct desired coherent sheaves.  
As usual, we first consider automorphic vector bundles on $X_\G$ and then 
extend them to a suitably chosen toroidal compactification of $X_\G$. 
Among them, we lose the ampleness of 
a natural automorphic line bundle $\omega$ on any toroidal compactification in most cases. 
However we can push forward the extended automorphic vector bundles in question to  
the minimal compactification preserving the coherence and descend  $\omega$ to an ample line bundle. 
In proving these things, we often use the results from reflexive sheaves, Serre's extension theorem 
(for the classical case), 
and later its variant over integral bases due to Kai-Wen Lan. 
Then the claim follows from a standard argument for coherent sheaves on projective schemes. 
Though the claim of Theorem \ref{cla} seems to be a folklore except for some cases or even a standard result in textbooks, 
it has been missing, such a finiteness might not have been clearly documented after Cartan Seminaires.  
Some important language is maintained during the last decade by Kai-Wen Lan and his collaborators though  
the classical modular forms on Shimura varieties have been understood very well in terms of various methods including  
keywords as $(\frak g,K)$-cohomology, mixed Hodge theory and so on.      

Next we consider similar claims for integral bases. 
To define an integral structure of the space of automorphic forms or cusp forms, let us assume that 
\begin{itemize}
\item $(G,\mathcal D)$ is a Shimura datum of PEL type. 
\end{itemize} 
For each rational prime $p$, 
let us fix an isomorphism $\iota_p:\bQ_p\stackrel{\sim}{\lra} \C$ and we say a subring $R\subset \C$ is $p$-adically integral if 
$\iota^{-1}_p(R)\subset \overline{\Z}_p$. We also say any subring of $\bZp$ $p$-adically integral. Note that $\Z_{(p)}$ and $\Z_p$ are standard examples of $p$-adically integral rings. For any positive integer $N$ and a finite extension $F/\Q$, 
the ring $\mathcal{O}_F[\frac{1}{N}]$ is also a $p$-adically integral ring when $p\nmid N$. 
 
Let $p$ be a good prime and $R_1$ be a $p$-adically integral ring defined in Subsection \ref{im}. If $G=GSp_{2g}/\Q$, then 
any rational prime $p$ is good and $R_1$ can be any of $\Z_{(p)}$ and $\Z_p$.  

Then it will be explained in next section that by using the moduli interpretation of $X_{\mathcal U}$ for 
any open compact subgroup $\U=U^p U_p$  such that $U^p$ is an open compact subgroup in $G(\widehat{\Z}^p)=\ds\prod_{q\neq p}G(\Z_q)$ and $U_p=G(\Z_p)$, 
for any $R_1$-algebra $R$, we can define geometric automorphic forms or geometric cusp forms over $R$. As in the cases before, 
we define the $R$-module 
$M_{\nu_0}(\U,R)$ (resp. $S_\rho(\U,R)$) consisting of automorphic forms (resp. cusp forms) over $R$ of weight 
$\nu_0\in X^+_{M_1}$ with respect to $\U$ (see \S\ref{im} for weights).  
According to this definition, as (\ref{def1}),(\ref{def2}), we also define the graded $R$-modules 
\begin{equation}\label{sps}
M_{\nu_0,\lambda,\ast}(\U,R)=
\ds\bigoplus_{k\in \Z_{\ge 0}}M_{\nu_0+k\lambda}(\U,R),\ S_{\nu_0,\lambda,\ast}(\U,R)=
\ds\bigoplus_{k\in \Z_{\ge 0}}S_{\nu_0+k\lambda}(\U,R),
\end{equation}
and 
\begin{equation}
M_{\lambda,\ast}(\U,R)=\ds\bigoplus_{k\in \Z_{\ge 0}}M_{k\lambda}(\U,R)
\end{equation}
where $\lambda\in X^+_{M_1}$ is a positive parallel weight in the sense of Definition 7.1, p.1153 of \cite{LS-duke}. 

\begin{thm}\label{arith1}Let $p$ be a good prime and $R_1$ be as above. Assume that the symmetric space $X$ is 
a Shimura variety of PEL type. Let $\nu_0\in  X^+_{M_1}$ be a weight and $\lambda$ be a positive parallel weight.  
Then it holds that 
\begin{enumerate}   
\item the graded ring $M_{\lambda,\ast}(\U,R)$ is finitely generated over $R$; 
\item the graded modules $M_{\nu_0,\lambda,\ast}(\U,R)$ and $S_{\nu_0,\lambda,\ast}(\U,R)$ are finitely generated over 
$M_{\lambda,\ast}(\U,R)$. 
\end{enumerate}
\end{thm}
Let $R_1$ be as above and $R$ be $R_1$-algebra which is $p$-adically integral in the above sense. 
For any finite $R$-module $M$ we denote by $M^{{\rm TF}}$ the maximal 
$R$-free quotient of $M$. 
We can also study a more finer structure on these $R$-modules:
\begin{cor}\label{cor1}Keep the notation in Theorem \ref{arith1}. Assume that an $R_1$-module $R$ is $p$-adically integral. Then it holds that  
\begin{enumerate}   
\item the graded ring $M_{\lambda,\ast}(\U,R)^{{\rm TF}}:=\ds\bigoplus_{k\in \Z_{\ge 0}}M_{k\lambda}(\U,R)^{{\rm TF}}$ is finitely generated over $R$; 
\item the $R$-free graded modules $M_{\nu_0,\lambda,\ast}(\U,R)^{{\rm TF}}:=
\ds\bigoplus_{k\in \Z_{\ge 0}}M_{\nu_0+k\lambda}(\U,R)^{{\rm TF}}$ and 
$S_{\nu_0,\lambda,\ast}(\U,R)
^{{\rm TF}}:=
\ds\bigoplus_{k\in \Z_{\ge 0}}S_{\nu_0+k\lambda}(\U,R)^{{\rm TF}}$ are finitely generated over 
$M_{\lambda,\ast}(\U,R)^{{\rm TF}}$.   
\end{enumerate}
Further, these objects give integral structures of the classical forms (\ref{def3}) respectively. 
\end{cor}
In the course of proving the main theorems, we will use the results of Kai-Wen Lan and his collaborators. 
However, if we focus on Siegel modular forms, we will have more finer results which will be explained in 
Section 4 by using only the results in \cite{CF}.

This paper will be organized as follows. 
In Section 2 we will built up the settings precisely and prove the main theorems. 
In Section 3 we give an explicit form of positive parallel weights and good prime in the 
case when $G=GSp_{2g}$. 
In the last section, we will prove the more finer version of the finiteness results for 
Siegel modular forms of level one. 

\textbf{Acknowledgments.}
The author would like to thank 
Nobuyoshi Takahashi for the useful discussion in algebraic geometry. He would also like to thank the Professor Siegfried B\"oecherer 
and Hirotaka Kodama for pushing him to prove the  
finiteness results in this paper. In particular, the paper \cite{Ko} inspired   
the problems in this paper in more general setting. The author would like to 
express  a special thank to Prof. Kai-Wen Lan for giving 
me some comments on an earlier version of this paper. 
Finally, the author thanks the referees for their helpful remarks and corrections.

\section{Settings and results}
\subsection{Shimura data and Shimura varieties}
Let us start recalling some basic facts of Shimura varieties. We refer \cite{Lan-example} and \cite{Milne-intro}. 
Let $(G,\D)$ be a Shimura datum introduced in Section \ref{intro} where $G$ is a connected reductive group over $\Q$ and 
$\D$ is the $G(\R)$-conjugacy classes of homomorphisms 
$h:\mathbb{S}:={\rm Res}_{\C/\R}\mathbb{G}_{m,\C}\lra G_\R$ 
which satisfy the following conditions:
\begin{enumerate}
\item The adjoint action of $G(\R)$ on the complexification $\frak g$ of 
the Lie algebra ${\rm Lie}\hspace{0.5mm} G(\R)$ and $h$ yield the homomorphism 
${\rm Ad}\circ h:\mathbb{S}(\R)=\C^\times\lra {\rm Aut}_\C(\frak g)$ and it induces a decomposition 
$$\frak g=\frak k\oplus \frak g^+\oplus \frak g^-$$
such that $h(z),\ z\in \C^\times$ acts on the right hand side of the above decomposition by $1,z/\overline{z},\overline{z}/z$ respectively.
\item $h(\sqrt{-1})$ induces a Cartan involution on $G^{{\rm ad}}(\R)$ where $G^{{\rm ad}}=G/Z_{G}$ and $Z_G$ is the center of $G$. 
\item $G^{{\rm ad}}$ has no nontrivial $\Q$-simple factor $H$ such that $H(\R)$ is compact.  
\end{enumerate}
Factoring through a connected Shimura variety (see Lemma 5.11 of \cite{Milne-intro}) the set $\D$ has a 
structure as a Hermitian symmetric domain (see Proposition 4.8 of \cite{Milne-intro}). 
The third condition for Shimura data guarantees that $G^{{\rm ad}}$ is semisimple. By Proposition 4.1 of \cite{Milne-intro},  
$G(\Q)\cap \U$ is a congruence subgroup for   
any compact open subgroup $\U$ and conversely any congruence subgroup is recovered in this way. 
As explained in Section \ref{intro}, the Shimura variety
$$X_\U:=G(\Q)\bs \D\times G(\A_f)/U\simeq \coprod_{i\in I}\G_i\bs \D^+$$ has 
a structure as a quasi-projective variety over $\C$ for any open compact subgroup $\U$ of $G(\A_f)$. We may work on $X_\G=\G\bs \D^+$ for 
any congruent subgroup $\G$ of $G(\Q)^+$ such as each connected component of $X_\U$. Since $\G$ has 
a finite index normal subgroup which is neat and the cohomologies in 
question are $\C$-vector spaces, we may assume that $\G$ is neat in proving Theorem \ref{cla} (cf. the argument around the equation (2) in Chapter IV p.140 of \cite{BW}). 
For an automorphic factor $J_\rho$ in (\ref{auto-factor}) one can associate the holomorphic automorphic vector bundle 
$W_\rho$ on $X_\G$ such that $H^0(X_\G,W_\rho)\simeq M_{\rho}(\G)$ (see Chapter III of \cite{Milne-canonical} or \cite{Serre-auto}). 

To apply some results on projective varieties we need to compactify $X_\G$ and canonically extend our sheaf $W_\rho$ there. 
Under this process the condition (dim) is reasonable to make no difference between holomorphic automorphic forms and 
holomorphic global sections of the extended coherent sheave.

Let $X_{\G,\Delta}$ be a smooth toroidal compactification of $X_\G$ with respect to a 
fan $\Delta$ (cf. Chapter 
V of \cite{Milne-canonical}). 
In fact, one can choose such a fan by using good cone decompositions. 
Then there exists a suitable choice of $\Delta$ such that 
$W_\rho$ extends to a vector bundle (so called a canonical extension) $W^{\rm can}_\rho$ on $X_{\G,\Delta}$ such that 
\begin{equation}\label{hkp}
H^0(X_{\G,\Delta},W^{\rm can}_\rho)\simeq M_\rho(\G),\ H^0(X_{\G,\Delta},W^{\rm sub}_\rho)\simeq S_\rho(\G)
\end{equation}
where $W^{\rm sub}_\rho=W^{\rm can}_\rho(-D_\Delta)$ and $D_\Delta=(X_{\G,\Delta}\setminus X_{\G})_{{\rm red}}$ 
(see Theorem 6.1 of \cite{Milne-canonical} for $W^{\rm can}_\rho$ and $W^{\rm sub}_\rho$).
The isomorphisms (\ref{hkp}) for cusp forms follow from Proposition 5.4.2 of \cite{Harris-JDG} when $\G$ is a congruence subgroup and 
Theorem 4.7 of \cite{Lan-van} for general case  since 
the codimension condition is fulfilled under the assumption (dim) (see  (\ref{G3})). 
In particular as mentioned in Remark 2.4 of \cite{Lan-HKP} 
the case of Siegel modular forms of degree greater than one, or Hilbert modular forms for totally real fields of degree greater than one 
 satisfies (dim). 
The interested readers for this condition may 
consult the table in Example 3.19 of \cite{Lan-coh}. 

To work on projective varieties it would be better to use toroidal compactifications rather than the minimal 
compactification. For example the former one is smooth while the latter one is normal and in general it has 
bad singularities. Nevertheless there are some advantages to work on  the minimal compactification to naturally get 
a suitable automorphic line bundle which is ample there but not on toroidal compactifications (see Section 3.1 of 
\cite{Lan-coh}). In fact we need to study (holomorphic) automorphic bundles on the minimal compactification to 
get the finiteness results. 
 
Let $j^{{\rm min}}:X_\G\hookrightarrow X^{{\rm min}}_\G$ be the minimal compactification. 
By the assumption (\ref{G3}), the codimention of $X^{{\rm min}}_\G\setminus X_\G$ in $X^{{\rm min}}_\G$ is greater than or 
equal to 2. Therefore, it follows from Serre's extension theorem \cite{Serre-pro} that  
$j^{{\rm min}}_\ast W_\rho$ is coherent (see the proof of Theorem 10.14 of \cite{BB}) and this is a unique extension of $W_\rho$ 
to $X^{{\rm min}}_\G$. Let $\pi:X_{\G,\Delta}\lra X^{{\rm min}}_\G$ be the canonical proper surjective morphism (see the proof of 
Lemma \cite{Lan-van}). Then $\pi_\ast W^{\rm can}_\rho$ and 
$\pi_\ast W^{\rm sub}_\rho$  are both coherent since $\pi$ is proper.
It follows from (\ref{hkp}) that 
\begin{equation}\label{mini-auto}
H^0(X^{{\rm min}}_\G, \pi_\ast W^{{\rm can}}_\rho)\simeq M_\rho(\G)
\end{equation} 
and 
\begin{equation}\label{mini-cusp}
H^0(X^{{\rm min}}_\G,\pi_\ast W^{\rm sub}_\rho)\simeq S_\rho(\G).
\end{equation}
Therefore, we have two coherent sheaves on the normal projective variety $X^{{\rm min}}_\G$ 
which give rise to automorphic forms and cusp forms respectively.  

Now we are ready to prove Theorem \ref{cla}. 
\begin{proof}
Assume that $\lambda$ is a positive character of rational type. 
Then by Theorem 1 of \cite{Satake-some} and Lemma 3.2 of \cite{Lan-van} the line bundle $W^{\rm can}_\lambda$ descend 
to the ample line bundle $\omega_\lambda$ which is nothing but $j^{{\rm min}}_\ast W_\lambda$. 
In fact since $\pi$ is proper birational and $ X^{{\rm min}}_\G$ is normal, 
by Zariski main theorem, $\pi_\ast \O_{X_{\G,\Delta}}=\O_{X^{{\rm min}}_\G}$. 
It follows from the projection formula that $\pi_\ast W^{\rm can}_\lambda=\pi_\ast \pi^\ast \omega_\lambda=\omega_\lambda$. 
Since  $\omega_\lambda$ and $j^{{\rm min}}_\ast W_\lambda$ are reflexive, and  
$\omega_\lambda |_{X_{\G}}=(\pi_\ast W^{\rm can}_\lambda)|_{X_{\G}}=W_\lambda=(j^{{\rm min}}_\ast W_\lambda)|_{X_{\G}},$
we have 
\begin{equation}\label{key}
\pi_\ast W^{\rm can}_\lambda=\omega_\lambda=j^{{\rm min}}_\ast W_\lambda
\end{equation} by Proposition 1.6, p.126 of \cite{H80} or Theorem 3, p.817 of \cite{Ghitza}.  
This is a key ingredient regarding the following cohomological description of automorphic forms and cusp forms. 

Since $H^0(X^{{\rm min}}_\G,j^{{\rm min}}_\ast W_{\lambda^k})\simeq M_{\lambda^k}(\G)$ for non-negative integer $k$, 
the graded ring $M_{\lambda,\ast}(\G)\simeq \bigoplus_{k\in \Z_{\ge 0}}H^0(X^{{\rm min}}_\G,j^{{\rm min}}_\ast W_{\lambda^k})$ 
is finitely generated by Lemma 16.1 of \cite{Stacks}. Similarly since 
\begin{equation}\label{isom1}
M_{\rho,\lambda,\ast}(\G)\simeq \bigoplus_{k\in \Z_{\ge 0}}H^0(X^{{\rm min}}_\G,
\pi_\ast W^{{\rm can}}_\rho\otimes 
(j^{{\rm min}}_\ast W_{\lambda})^{\otimes k}) 
\end{equation} by (\ref{mini-auto}), (\ref{key}) and  
\begin{equation}\label{isom2}
S_{\rho,\lambda,\ast}(\G)\simeq \bigoplus_{k\in \Z_{\ge 0}}H^0(X^{{\rm min}}_\G,\pi_\ast W^{\rm sub}_\rho\otimes 
(j^{{\rm min}}_\ast W_\lambda)^{\otimes k})
\end{equation}
by (\ref{mini-cusp}), (\ref{key}), 
the claim for these graded vector spaces follows from Lemma 16.1-(5) of \cite{Stacks}.   
\end{proof}

\subsection{Integral models}\label{im} 
In this subsection we impose the following condition to work on $p$-adically integral bases:
\begin{equation}\label{pel}
\mbox{(PEL) $(G,\D)$ is a Shimura datum of PEL type.}
\end{equation}
We refer \cite{L-book}, \cite{LS-rel}, \cite{LS-adv}, \cite{LS-duke}, \cite{LS-lift} and also \cite{Lan-geo}. 
Since the notations in \cite{LS-adv},\cite{LS-duke} would be heavy for most readers, 
we recall the results in  \cite{LS-rel} quickly and avoid explaining in detail. Instead we give a few examples 
which would be enough for applications to many interesting cases. 
By classification any irreducible factor of $(G,\D)$ is of type A, C, or D.  
According to Section 1.1 of \cite{LS-duke}, let us consider 
an integral PEL datum $(\mathcal{O},\star,L,\langle \ast,\ast \rangle,h_0 )$  in the following sense:
\begin{enumerate}\label{def-pel}
\item $\mathcal{O}$ is an order in a non-zero semisimple algebra, finite dimensional over $\Q$ after 
tensoring with $\Q$, 
together with a positive involution $\star$;
\item $(L,\langle \ast,\ast \rangle,h_0)$ is a PEL type $\mathcal{O}$-lattice 
(a polarized symplectic lattice in other word, cf. Definition 1.2.1.2 of \cite{L-book})  
\end{enumerate}
Let $F$ be the center of $\mathcal{O}\otimes_\Z\Q$ which is a product of number fields. 
Then we define for any $\Z$-algebra $R$, 
$$G(R):=\{(g,r)\in {\rm GL}_{\mathcal{O}\otimes_\Z R}(L\otimes_\Z R)
\times \mathbb{G}_m(R)\ |\ \langle gx,gy \rangle=r \langle x,y \rangle,\ 
\forall x,y\in L\otimes_\Z R \}.$$
As explained in Remark 1.2.1.8, the group functor is not necessarily 
a smooth functor over $\Z$ but an affine group scheme over $\Z$. 
However one can easily check that $G_\Q$ is a smooth reductive group over $\Q$ and 
it is also connected because of the similitude character. 

The polarization $h_0:\C\lra {\rm End}_{\O\otimes_\Z R}(L\otimes_\Z\R)$ defines a Hodge structure of 
weight $-1$ with Hodge decomposition $L\otimes_\Z\C=V_0\oplus V^c_0$ as a $\O\otimes_\Z\C$-module, 
such that $h_0(z)$ acts as $1\otimes z$ on $V_0$ and as $1\otimes z^c$ on $V^c_0$. 
Here superscript ``$c$" stands for the complex conjugation. Let $F_0$ be the reflex field 
defined by the $\O\otimes_\Z\C$-module $V_0$ (see Section 1.2.5.4, p.51 of \cite{L-book}). 
For instance, $F_0=\Q$ if $G={\rm Res}_{F/\Q}GSp_{2n,F}$ for a totally real field $F$ and in general 
$F_0$ is a subfield of the Galois closure of $K$ if  $G$ is the unitary similitude group $GU(p,q)=GU(p,q)(K/F)$ for 
a CM extension $K/F$ but $F_0=\Q$ if we further assume $p=q$. 
To be more precise when $K/\Q$ is an imaginary quadratic extension, $F_0=K$ unless $p=q$ (see 
Chapter III, Section 1, p.143 of \cite{Harris-unitaryI}).  
 For Example 5.24, p.312 and Example 12.4-(d), p. 344 (which is related to Shimura curves) of \cite{Milne-intro} we have $F_0=v(F)$ when 
$I_{{\rm nc}}=\{v\}\subset {\rm Hom}_{\Q}(F,\R)$ 
for a totally real field (see \cite[Example 5.24, p.312]{Milne-intro} for the symbol $I_{{\rm nc}}$). 

Let ${\rm Diff}^{-1}$ be the inverse difference of $\O/\Z$ and put ${\rm Disc}=[{\rm Diff}^{-1}:\O]$ 
(see (1.1.1.17), p.4 of \cite{L-book}). 
We say a rational prime $p$ is good if it satisfies
\begin{enumerate}
\item $p\nmid {\rm Disc}$;
\item $p\neq 2$ if $\O_\Z\otimes\Q$ involves a simple 
factor of type $D$, in the sense of Definition 1.2.1.15, p.31 of \cite{L-book}; 
\item the pairing $\langle \ast,\ast \rangle$ is perfect after the base change to $L\otimes_\Z\Z_p$. 
This is equivalent to ask if $p\nmid [L^\sharp:L]$ for the dual lattice $L^\sharp$ of $L$.  
\end{enumerate}
For $GSp_{2n}/F$ or $GU(n,n)=GU(n,n)(K/F)$ (the similitude unitary groups for a CM extension $K/F$), Disc is nothing but the discriminant of $F$ or $K$ and 
$p$ is a good prime if and only if $p$ is unramified in $F$ or $K$ respectively. 

By Lemma 1.2.5.9, p.52 of \cite{L-book}, there exists a finite extension $F'_0$ of $F_0$ in $\C$, unramified at $p$, 
together with an $\O\otimes_{\Z}\O_{F'_0,(p)}$-module $L_0$ such that $L_0\otimes_{\O_{F'_0,(p)}}\C\simeq 
V_0$ as a $\O\otimes_\Z\C$-module. Here $\O_{F'_0,(p)}$ stands for the localization of $\O_{F'_0}$ with respect to  
the ideal $(p)$. 
One can easily find $F'_0$ out from the statement or the proof of above lemma. 
For instance, if $F$ or $K$ is Galois for an integral PEL datum in the case of   
$GSp_{2n}/F$ or $GU(n,n)(K/F)$ as above, then $F'_0=F_0$. 

For a good prime $p$ and $F'_0$, put $W_0=L_0\oplus L^\vee_0(1)$ and let us denote by $\langle \ast,\ast \rangle_{{\rm can}}:W_0
\times W_0\lra \O_{F'_0,(p)}(1)$ the alternating pairing defined in Lemma 1.1.4.13, p.20 of \cite{L-book}. 
This is an integral structure of Hodge decomposition $L\otimes_\Z\C=V_0\oplus V^c_0$. 
We define an integral model of $G$ over $\O_{F'_0,(p)}$ as follows. 
For any $\O_{F'_0,(p)}$-algebra $R$ set 
$$G_0(R)=\{(g,r)\in GL_{\O\otimes_\Z R}(W_0\otimes_{\O_{F'_0,(p)}}R)\times \mathbb{G}_m(R)\ |\ 
\langle gx,gy \rangle_{{\rm can}}=r \langle x,y \rangle,\ \forall x,y \in 
W_0\otimes_{\O_{F'_0,(p)}}R \}.$$
Similarly the Siegel parabolic subgroup $P_0$ of $G_0$ and its Levi factor $M_0$ which are both defined over $\O_{F'_0,(p)}$ 
are given in Definition 1.4 of \cite{LS-duke}. 
By line -14, p.1117 of \cite{LS-duke} there exists a discrete valuation ring $R_1$ over $\O_{F'_0,(p)}$ 
satisfies the conditions (1),(2),(3) there. This relates the original $\langle \ast,\ast \rangle$ with 
$\langle \ast,\ast \rangle_{{\rm can}}$ over $R_1$. Hence $G_0\times_{\O_{F'_0,(p)}}R_1\simeq G_{R_1}.$
This is necessary to define an integral automorphic vector bundle over $R_1$ which will be revealed later on. 
As for $R_1$, one can take $R_1$ to be the localization of $\O_{F'_0,(p)}$  at a prime ideal dividing $(p)$ 
when $GSp_{2n}/F$ or $GU(n,n)(K/F)$ as above. Hence for a prime ideal  $v$ dividing $(p)$ in 
$\O_{F'_0,(p)}=\O_{F,(p)}$ or $\O_{K,(p)}$, one can take $R_1=\O_{F,(v)}$ or $\O_{K,(v)}$ respectively. 
We can also consider its $v$-adic completion $\O_{F,v}$ or $\O_{K,v}$ as $R_1$. In particular if $G=GSp_{2n}/\Q$, then 
$R_1$ can be $\Z_{(p)}$ or $\Z_p$. 

Let us fix $R_1$ and set $$G_1:=G_0\times_{\O_{F'_0,(p)}}R_1,\ P_1:=P_0\times_{\O_{F'_0,(p)}}R_1,\ M_1:=M_0\times_{\O_{F'_0,(p)}}R_1.$$ 

Since the polarization $h_0$ is a $\R$-algebra homomorphism, it is determined by $h_0(\sqrt{-1})$ and it also defines an 
element in $G(\R)$ by Definition 1.2.1.2-1 of \cite{L-book}. Hence $(G,G(\R)h_0)$ define a Shimura datum where 
$G(\R)$ acts on $h_0$ by conjugation. As in the previous subsection it gives rise to the 
Shimura variety $X_\U$ for any $\U=U^pU_p$ where $U^p$ is an open compact subgroup of $G(\widehat{\Z}^p)$ and 
$U_p=G(\Z_p)$. Here $\widehat{\Z}^p=\ds\varprojlim_{N,\ p\nmid N}\Z/N\Z$. 
Assume that $U^p$ is neat. As in Section 1.2 of \cite{LS-duke}, the PEL-datum $(\O,\star,L,
\langle \ast,\ast \rangle,h_0)$ and $U^p$ define a moduli problem $M_{U^p}$ over $S_0:={\rm Spec}\hspace{0.5mm}\O_{F_0,(p)}$, 
parameterizing tuples $(A,\lambda,i,\alpha_{U^p})$ over $S_0$-schemes $S$ of the following form:
\begin{enumerate}
\item $A\lra S$ is an abelian scheme;
\item $\lambda:A\lra A^\vee$ is a polarization of degree prime to $p$;
\item $\underline{{\rm Lie}}_{A/S}$ with its $\O\otimes_\Z \Z_{(p)}$-module structure given naturally by $i$ 
satisfies the (Kottwitz) determinantal condition in Definition 1.3.4.1, p.69 of \cite{L-book};
\item $\alpha_{U^p}$ is an integral level $U^p$-structure of $(A,\lambda,i)$ of type $(L\otimes_\Z \widehat{\Z}^p,
\langle \ast,\ast \rangle)$ as in Definition 1.3.6.2,\ p.72 of \cite{L-book}.  
\end{enumerate}  
Then by Theorem 1.4.1.12, p.82 and Corollary 7.2.3.10, p.461 of \cite{L-book}, the moduli problem 
$M_{U^p}$ is represented by a smooth quasi-projective scheme over $S_0$. By Section 2 of \cite{L-comp} 
there is a canonical open and closed immersion 
$X_{\U}\hookrightarrow M_{U^p}\otimes_{\O_{F_0,(p)}}F_0$ which is defined over $F_0$. 
Let $\mathcal{X}_{\U}$ be the schematic closure of $X_{\U}$ in $M_{U^p}$ (it is written by $M_{\mathcal{H},0}$ for 
$\mathcal{H}=U^p$ in the notation of \cite{LS-duke}). 
By Proposition 4.2, p.250 of \cite{LS-adv} $\mathcal{X}_{\U}$ admits a toroidal compactification 
$\mathcal{X}^{{\rm tor}}_{\U}=\mathcal{X}^{{\rm tor}}_{\U,\Sigma}$, a scheme projective and smooth over $S_0={\rm Spec}\hspace{0.5mm}R_1$ 
depending on a cone decomposition $\Sigma$.  
  
In what follows we refer Section 1, 2 of \cite{LS-duke} for (integral) automorphic vector bundles and we 
follow the notation there.  
For any $\nu_0\in X^+_{M_1}$ and any $R_1$-algebra $R$ we can define the automorphic vector bundle $\underline{W}_{\nu_0,R}$ over $X_{\U}$. By Lemma 1.18 of \cite{LS-duke} and by definition it is locally free sheaf on $X_{\U}$.  
As explained in Section 4 of \cite{LS-adv} we can define the canonical extension $\underline{W}^{{\rm can}}_{\nu_0,R}$ 
and subcanonical extension $\underline{W}^{{\rm sub}}_{\nu_0,R}=
\underline{W}^{{\rm can}}_{\nu_0,R}\otimes \mathcal{I}_D$ where $\mathcal{I}_D$ is the $\O_{\mathcal{X}^{{\rm tor}}_{\U}}$-ideal 
defining relative Cartier Divisor $(\mathcal{X}^{{\rm tor}}_{\U}\setminus \mathcal{X}_{\U})_{{\rm red}}$.   
   
Then the space of geometric automorphic forms (resp. geometric cusp forms) over $R$ of weight $v_0$ with respect to $\U$ are defined by 
\begin{equation}\label{geo}
M_{v_0}(\U,R):=H^0(\mathcal{X}^{{\rm tor}}_{\U,R},\underline{W}^{{\rm can}}_{\nu_0,R}),\ 
S_{v_0}(\U,R):=H^0(\mathcal{X}^{{\rm tor}}_{\U,R},\underline{W}^{{\rm sub}}_{\nu_0,R}).
\end{equation} 
Let $\lambda \in X^+_{M_1}$ be a positive parallel weight in the sense of Definition 7.1, p.1153 of \cite{LS-duke}. 
Then define the graded vector space of geometric automorphic forms 
\begin{equation}\label{geo-auto}
M_{v_0,\lambda,\ast}(\U,R):=\bigoplus_{k\in \Z_{\ge 0}}H^0(\mathcal{X}^{{\rm tor}}_{\U,R},\underline{W}^{{\rm can}}_{\nu_0+k \lambda,R})
=\bigoplus_{k\in \Z_{\ge 0}}H^0(\mathcal{X}^{{\rm tor}}_{\U,R},\underline{W}^{{\rm can}}_{\nu_0,R}
\otimes_{\O_{\mathcal{X}^{{\rm tor}}_{\U,R}}}(\underline{W}^{{\rm can}}_{\lambda,R})^{\otimes k})
\end{equation}
and the graded vector space of geometric cusp forms 
\begin{equation}\label{geo-cusp}
S_{v_0,\lambda,\ast}(\U,R):=\bigoplus_{k\in \Z_{\ge 0}}H^0(\mathcal{X}^{{\rm tor}}_{\U,R},\underline{W}^{{\rm sub}}_{\nu_0+k \lambda,R})
=\bigoplus_{k\in \Z_{\ge 0}}H^0(\mathcal{X}^{{\rm tor}}_{\U,R},\underline{W}^{{\rm sub}}_{\nu_0,R}
\otimes_{\O_{\mathcal{X}^{{\rm tor}}_{\U,R}}}(\underline{W}^{{\rm can}}_{\lambda,R})^{\otimes k}).
\end{equation}
We also define 
$M_{\lambda,\ast}(\U,R)=M_{\textbf{0},\lambda,\ast}(\U,R)$ and 
$S_{\lambda,\ast}(\U,R)=S_{\textbf{0},\lambda,\ast}(\U,R)$ 
where $\textbf{0}$ stands for the trivial element in $X^+_{M_1}$.   
  
As in the classical case, we try to relate theses spaces with coherent sheaves on the minimal compactification. 
An algebraic model $\mathcal{X}^{{\rm min}}_{\U}$ of the minimal compactification $X^{{\rm min}}_{\U}$ is constructed in 
Chapter 7 of \cite{L-book} and it is a normal projective scheme over ${\rm Spec}\hspace{0.5mm}\O_{F_0,(p)}$ together with 
a proper surjective birational morphism $\pi:\mathcal{X}^{{\rm tor}}_{\U}\lra \mathcal{X}^{{\rm min}}_{\U}$ which commutes with 
the natural embedding $j^{{\rm min}}:\mathcal{X}_\U\lra \mathcal{X}^{{\rm min}}_{\U}$ and 
$j^{{\rm tor}}:\mathcal{X}_\U\lra \mathcal{X}^{{\rm min}}_{\U}$. 
It is well-known that the line bundle $\underline{W}_{\lambda,R}$ is obtained by the pullback of 
an ample line bundle $L_{\lambda,R}$ on $\mathcal{X}^{{\rm min}}_{\U,R}$ via $\pi$ (see Section 2A of \cite{LS-rel}). 
Now we are ready to prove Theorem \ref{arith1}. The situation is easier than the classical case. 
\begin{proof}
It follows from the definition of direct images of sheaves that 
\begin{equation}\label{geo-auto2}
M_{v_0,\lambda,\ast}(\U,R)=
\bigoplus_{k\in \Z_{\ge 0}}H^0(\mathcal{X}^{{\rm min}}_{\U,R},\pi_\ast \underline{W}^{{\rm can}}_{\nu_0,R}
\otimes_{\O_{\mathcal{X}^{{\rm min}}_{\U,R}}}L_{\lambda,R}^{\otimes k})
\end{equation} 
and  
\begin{equation}\label{geo-cusp2}
S_{v_0,\lambda,\ast}(\U,R)
=\bigoplus_{k\in \Z_{\ge 0}}H^0(\mathcal{X}^{{\rm min}}_{\U,R},\pi_\ast \underline{W}^{{\rm sub}}_{\nu_0,R}
\otimes_{\O_{\mathcal{X}^{{\rm min}}_{\U,R}}}L_{\lambda,R}^{\otimes k}).
\end{equation} 
Notice that  $\pi_\ast \underline{W}^{{\rm can}}_{\nu_0,R}$ and $\pi_\ast \underline{W}^{{\rm sub}}_{\nu_0,R}$ 
are coherent, since $\pi$ is proper. The claim follows from Lemma 16.1-(1),(5) of \cite{Stacks}.   
\end{proof}

Next we give a proof of Corollary \ref{cor1}. 
\begin{proof}We may assume that $R$ is a DVR by flat base change. 
Let $\kappa_R=R/m_R$ where $m_R$ is the maximal ideal of $R$. 
By Serre's vanishing theorem, for any $i>0$ and $k\gg 0$  
$$H^i(\mathcal{X}^{{\rm min}}_{\U,\kappa_R},\pi_\ast \underline{W}^{{\rm can}}_{\nu_0,\kappa_R}
\otimes_{\O_{\mathcal{X}^{{\rm min}}_{\U,\kappa_R}}}L_{\lambda,R}^{\otimes k})=0$$
and 
$$H^i(\mathcal{X}^{{\rm min}}_{\U,\kappa_R},\pi_\ast \underline{W}^{{\rm sub}}_{\nu_0,\kappa_R}
\otimes_{\O_{\mathcal{X}^{{\rm min}}_{\U,\kappa_R}}}L_{\lambda,R}^{\otimes k})=0.$$
For such a $k$, the argument in the proof of Corollary 4.3, p.1877 of \cite{LS-lift} shows that 
$M_{\nu_0+k\lambda}(\U,R)$ and $S_{\nu_0+k\lambda}(\U,R)$ are free over $R$. 
Take any non-negative integer $k$ such that $\nu_0+k\lambda$ satisfies the above vanishing for higher cohomology.  
The product induces a paring 
$$M_{\nu_0}(\U,R)\times M_{k\lambda}(\U,R)\lra M_{\nu_0+k\lambda }(\U,R),\ (f,g)\mapsto f\cdot g.$$
Since $M_{\nu_0+k\lambda }(\U,R)$ is torsion free, this paring factors through 
$M_{\nu_0}(\U,R)^{{\rm TF}}\times M_{k\lambda}(\U,R)^{{\rm TF}}$. 
It is the same for cusp forms. 
 Hence we have the decomposition 
$$M_{v_0,\lambda,\ast}(\U,R)=T_1\oplus\Big(\bigoplus_{k\in \Z_{\ge 0}} M_{v_0+k\lambda}(\U,R)^{{\rm TF}}\Big),\ S_{v_0,\lambda,\ast}(\U,R)=T_2\oplus\Big(\bigoplus_{k\in \Z_{\ge 0}} S_{v_0+k\lambda}(\U,R)^{{\rm TF}}\Big),$$ 
and $M_{\lambda,\ast}(\U,R)=T_3\oplus\Big(\bigoplus_{k\in \Z_{\ge 0}} M_{k\lambda}(\U,R)^{{\rm TF}}\Big)$ 
where $T_1,T_2,T_3$ are torsion $R$-modules which are 
also finitely generated over $R$. Hence the claim follows from 
Theorem \ref{arith1}. 
\end{proof}

\section{An application to Siegel modular forms}\label{app} 
Let us consider $G=GSp_{2g}/\Z$ with the similitude character $\nu:G\lra \mathbb{G}_m$. Its derived group 
$G^{{\rm der}}=Sp_{2g}={\rm Ker}(\nu)$ is a semisimple reductive group scheme of type $(C)$. 
It naturally gives a Shimura datum of a PEL type (see Subsection 3.1 of \cite{Lan-example}). 
Put $\G={\rm Sp}_{2g}(\Z)$ and for any positive integer $N$, 
we denote by $\G(N)$ be the principal congruence subgroup of level $N$. 
Let $K(N)$ be the open compact subgroup of $G(\widehat{\Z})$ consisting of all elements which are 
congruent to the identity element modulo $N$. It follows that $K(N)\cap {\rm Sp}_{2g}(\Q)=\G(N)$ and $K(N)$ is neat if $N\ge 3$.   
%Put $d_N=|\G/\G(N)|=|K(1)/K(N)|=N^{2g^2+g}\ds\prod_{p|N}\prod_{i=1}^g(1-p^{-2i})$. 
Any element $\nu_0$ of $X^+_{M_1}$ 
can be written by a tuple $((k_1,\ldots,k_g);k_0))$ where  $k_1\ge k_2\ge\cdots \ge k_g$ and $k_0$ are integers. 
In view of the application here the last entry $k_0$ is unnecessary and it will play an important role when 
we relate classical forms with adelic forms though we do not pursue it. 
Then we have that 
\begin{enumerate}
\item any rational prime $p$ is good;  
\item $R_1$ can be any of $\Z_{(p)}$ and $\Z_p$; 
\item any positive parallel weight can be represented by $k_1=\cdots=k_g\ge 1$ (see Lemma 3.49, p.13 of \cite{Lan-van}).  
\end{enumerate}
For any $\underline{k}=(k_1,\ldots,k_g)\in \Z^g$ satisfying $k_1\ge \cdots \ge k_g$ and an 
arithmetic subgroup $\G$ of ${\rm Sp}_{2g}(\Q)$, 
we denote by $M_{\underline{k}}(\G)$ (resp. $S_{\underline{k}}(\G)$) the space of Siegel modular forms 
(resp. Siegel cusp forms) of 
weight $\underline{k}$ with respect to $\G$. 
Assume that $\G$ contains $\G(M)$ for some $M$ as a finite index subgroup and 
put $d_M:=[\G:\G(M)]$. 
Since $M_{\underline{k}}(\G)\subset M_{\underline{k}}(\G(M))$ and $S_{\underline{k}}(\G)\subset S_{\underline{k}}(\G(M))$, 
by using $q$-expansion principle (cf. Theorem 2 of \cite{Ichikawa}), 
for any subring $R$ of $\C$, we define the space 
$M_{\underline{k}}(\G,R)$ consisting of all Siegel modular forms over $\C$ whose Fourier coefficients at the cusp with respect to 
the Siegel parabolic subgroup are defined over $R$. Similarly we can define $S_{\underline{k}}(\G,R)$. 
\begin{thm}Let $\uk$ be as above. Let $K$ be an open compact subgroup of $G(\widehat{\Z})$ such that 
$\nu(K)=\widehat{\Z}^\times$. Put $\G_K={\rm Sp}_{2g}(\Q)\cap K$ and 
assume $\G_K$ contains $\G(M)$ for some positive integer $M\ge 3$. Let $p$ be a 
rational prime 
which never divides $d_M$. Let $R_1$ be $\Z_{(p)}$ or $\Z_p$. 
Then, it holds that  
\begin{enumerate}
\item $M_\ast(\G_K,R_1):=\ds\bigoplus_{k\in \Z_{\ge 0}}M_{k\textbf{1}}(\G,R_1)$ is finitely generated over $R_1$; 
\item $M_{\underline{k},\ast}(\G_K,R_1):=\ds\bigoplus_{k\in \Z_{\ge 0}}M_{\underline{k}+k\textbf{1}}(\G_K,R_1)$ 
and $S_{\underline{k},\ast}(\G_K,R_1):=\ds\bigoplus_{k\in \Z_{\ge 0}}M_{\underline{k}+k\textbf{1}}(\G_K,R_1)$ are 
finitely generated over $M_\ast(\G_K,R_1)$. 
\end{enumerate}
\end{thm}
\begin{proof}Let us consider the finite group $G=K/K(M)$ whose cardinality is coprime to $p$ by assumption. 
Then we have $M_{\underline{k}+k\textbf{1}}(\G_K,R_1)=M_{\underline{k}+k\textbf{1}}(K(M),R_1)^G$ and it is the same for cusp forms. 
The claim follows from Corollary \ref{cor1}. 
\end{proof}

%\subsection{Hermitian modular forms for imaginary quadratic fields}\label{HMF}
%We refer Subsection 4.2.1 of \cite{Hida} for automorphic factors, Section 3 of \cite{Gu}, and \cite{Harris-unitaryI}. 
%Let us consider the similitude unitary group  $G=GU(p,q)(K/\Q)$ defined by 
%$I_{p,q}:=
%\left(
%\begin{array}{cc}
%I_p & 0 \\
%0   & -I_q
%\end{array}
%\right)
%$ for an imaginary quadratic extension. Then $K_\C\simeq GL_p\times GL_q\times GL_1$. 
%Hence the highest weights of finite dimensional representation of $K_\C$ are the 
%$(p+q+1)$-tuples 
%$$\nu_0:=((a_1,\ldots,a_p,b_1,\ldots,b_q);c),\ a_i,b_j,c\in \Z,\ 1\le i\le p,\ 1\le j\le q$$
%such that 
%$$a_1\ge \cdots \ge a_p,\ b_1\ge \cdots \ge b_q,\ \sum_{i=1}^pa_i+\sum_{j=1}^qb_j\equiv c\ {\rm mod}\ 2.$$

\section{Another classical setting}
In this section, we will discuss the previous claims for the Siegel modular forms of level one by using only Chai-Faltings's results in \cite{CF}. 
Let us keep the notation in the previous section. Assume that $g\ge 2$. 
Let $k_1\ge \cdots \ge k_g$ be integers. Put $\uk=(k_1,\ldots,k_g)$. 
Let $\rho=\rho_{\uk}:GL_g(\C)\lra {\rm Aut}_\C(V_\rho)$ be a unique irreducible representation with 
the highest weight $\uk$. In terms of classical language, as in \cite{vG} we can define the 
space $M_\rho({\rm Sp}_{2g}(\Z))$ (resp. $S_\rho({\rm Sp}_{2g}(\Z))$) consisting of Siegel modular forms on 
the Siegel upper half space  $\mathbb{H}_g$ 
(resp. Siegel cusp forms) of weight $\rho$ with respect to ${\rm Sp}_{2g}(\Z)$. 
It follows from  Theorem 2.3-(2) of \cite{CF} that the graded ring 
$M_{\ast}({\rm Sp}_{2g}(\Z),\Z):=\ds\bigoplus_{k\in \Z_{\ge 0}}M_{k\textbf{1}}({\rm Sp}_{2g}(\Z),\Z)$ 
is finitely generated over $\Z$. Here $M_{k\textbf{1}}({\rm Sp}_{2g}(\Z),\Z)$ is the subspace of 
$M_{k\textbf{1}}({\rm Sp}_{2g}(\Z))$ consisting of all forms with integral Fourier coefficients. 

There is no canonical way to define geometric Siegel modular forms of level one. 
To detour this issue we can apply the results in \cite{CF} in terms of stacks. 
However to save notation and to avoid using much of stacks, we work on schemes and use reflexive-ness of some coherent sheaves related to Siegel modula forms.  
The coarse moduli $\mathcal{A}_g={\rm Sp}_{2g}(\Z))\bs\mathbb{H}_g$ is not a complex manifold but an orbifold since ${\rm Sp}_{2g}(\Z)$ has non-trivial torsions. 
However there is a model $\mathcal{A}_{g,\Z}$ which is a quasi-projective normal scheme over $\Z$
such that $\mathcal{A}_{g,\Z}(\C)\simeq \mathcal{A}_g$ as an analytic space  (see \cite[Theorem 2.3, p.150]{CF}). 
By \cite[Theorem 2.3, p.150]{CF} again, there also exists a canonical compactification  $\mathcal{A}^{{\rm min}}_{g,\Z}$  of  
$\mathcal{A}_{g,\Z}$ (which is so called 
the minimal compactification) and an ample line bundle $\mathcal{L}$ on it. 
As claimed there, $\mathcal{A}^{{\rm min}}_{g,\Z}$ is a projective normal scheme over $\Z$. 
Let $\mathcal{A}^{{\rm reg}}_{g,\Z}$ be the regular locus of $\mathcal{A}_{g,\Z}$. 
Since $\mathcal{A}^{{\rm min}}_{g,\Z}$ is normal and the codimension of 
$\mathcal{A}^{{\rm min}}_{g,\Z}\setminus \mathcal{A}_{g,\Z}$ in $\mathcal{A}^{{\rm min}}_{g,\Z}$ is greater or equal to 
$\ds\frac{g(g+1)}{2}-\ds\frac{g(g-1)}{2}=g \ge 2$, so is for  the codimension of 
$\mathcal{A}^{{\rm min}}_{g,\Z}\setminus \mathcal{A}^{{\rm reg}}_{g,\Z}$ in 
$\mathcal{A}^{{\rm min}}_{g,\Z}$.  
Let $f:X'\lra \mathcal{A}^{{\rm reg}}_{g,\Z}$ be the universal abelian variety which is a morphism of schemes over $\Z$. 
This will be defined as follows. We first consider the universal abelian variety $X\lra 
 [\mathcal{A}_{g,\Z}]$ over the smooth stack $[\mathcal{A}_{g,\Z}]$ (see \cite[Theorem 6.7, p.130]{CF} 
 for the properties of $[\mathcal{A}_{g,\Z}]$). 
Since $\mathcal{A}_{g,\Z}$ is a corase moduli scheme of $[\mathcal{A}_{g,\Z}]$ (\cite[Theorem 2.3-(3), p.150]{CF}), 
there is a natural map $\alpha:\mathcal{A}_{g,\Z}\lra [\mathcal{A}_{g,\Z}]$ as a stack (see also \cite[Theorem 4.10, p.23]{CF}). 
Then, the map $f$, as a morphism of stacks, is defined to be the fiber product of $X\lra 
 \mathcal{A}_{g,\Z}$ and $\alpha|_{\mathcal{A}^{{\rm reg}}_{g,\Z}}$. Then, it yields the desired morphism 
$f:X'\lra \mathcal{A}^{{\rm reg}}_{g,\Z}$ by Example 5.1.7, p.121 of \cite{Olsson}.      
Since $f$ is smooth by the fiber-wise argument, 
$\E=f_\ast \Omega^1_{X/\mathcal{A}^{{\rm reg}}_{g,\Z}}$ is a locally free sheaf of rank $g$ and it is clearly  
reflexive. Let $\rho:GL_g\lra {\rm Aut}_g(V_\rho)$ be an irreducible algebraic representation. 
Since Young symmetrizers in Weyl's construction of $\rho$ are defined over $\Z[\frac{1}{g!}]$ 
(we just observe the denominators of Young symmetrizers), 
the representation $\rho$ is defined over $\Z[\frac{1}{g!}]$. 
This is not optimal, for example, the determinant character is defined over $\Z$ for any $g$.    
Let $R_\rho$ be the minimal subring of $\Z[\frac{1}{g!}]$ such that $\rho$ is defined.  
For each $\rho$, we can associate a locally free sheaf $\E_\rho$ on $\mathcal{A}^{{\rm reg}}_{g,R_\rho}:=
\mathcal{A}^{{\rm reg}}_{g,\Z}\times_\Z R_\rho$ such that  $\E_\rho$ is locally isomorphic to 
$V_\rho(R_\rho)\otimes_{R_\rho}\mathcal{O}_{\mathcal{A}^{{\rm reg}}_{g,R_\rho}}$.   
By Theorem 3, p.817 of \cite{Ghitza}, there exists unique extensions $\widetilde{\E}_\rho$ and $\E^{{\rm min}}_\rho$ of $\E_\rho$ on 
$\mathcal{A}_{g,R_\rho}:=\mathcal{A}_{g,\Z}\times_\Z R_\rho$ and $\mathcal{A}^{{\rm min}}_{g,R_\rho}:=
\mathcal{A}^{{\rm min}}_{g,\Z}\times_\Z R_\rho$ respectively. They are both coherent sheaves such that 
$$H^0(\mathcal{A}_{g,R_\rho},\widetilde{\E}_\rho)=H^0(\mathcal{A}^{{\rm reg}}_{g,R_\rho},\E_\rho)
=H^0(\mathcal{A}^{{\rm min}}_{g,R_\rho},\E^{{\rm min}}_\rho).$$ 
Note that $\E^{{\rm min}}_\rho$ is also a unique extension of $\widetilde{\E}_\rho$. 
By construction the ample line bundle $\mathcal{L}$ is a unique extension of the line bundle $\omega:=\det \E$. We denote by $\widetilde{L}$ a 
unique extension of $\omega$ on $\mathcal{A}_{g,\Z}$. 
Then we also have 
\begin{equation}\label{imp-Siegel}
H^0(\mathcal{A}_{g,R_\rho},\widetilde{\E}_\rho\otimes \widetilde{\mathcal{L}}^{\otimes k})=
H^0(\mathcal{A}^{{\rm reg}}_{g,R_\rho},\E_\rho\otimes \omega^{\otimes k})
=H^0(\mathcal{A}^{{\rm min}}_{g,R_\rho},\E^{{\rm min}}_\rho\otimes\mathcal{L}^{\otimes k}).
\end{equation}
By flat base change (Proposition 9.3 in Chapter III of \cite{H}), GAGA \cite{Serre}, and 
Serre's extension theorem (see the proof of Theorem 10.14 of \cite{BB}), we have  
\begin{equation}\label{base-change}
M_{\rho\otimes\det^k}({\rm Sp}_{2g}(\Z))\simeq 
H^0(\mathcal{A}_{g,R_\rho},\widetilde{\E}_\rho\otimes \widetilde{\mathcal{L}}^{\otimes k})\otimes \C
=H^0(\mathcal{A}^{{\rm min}}_{g,R_\rho},\E^{{\rm min}}_\rho\otimes\mathcal{L}^{\otimes k})\otimes \C.
\end{equation}
We need to compare $M_{\rho\otimes\det^k}({\rm Sp}_{2g}(\Z),R_\rho)$ with $H^0(\mathcal{A}_{g,R_\rho},\widetilde{\E}_\rho\otimes \widetilde{\mathcal{L}}^{\otimes k})^{{\rm TF}}$. 
For each integer $N\ge 1$ there exists a quasi-projective model $\mathcal{A}_{g,N}$ over $\Z$ of 
$\G(N)\bs \mathbb{H}_g$ such that $\mathcal{A}_{g,N}$ is smooth over $R_N:=\Z[\frac{1}{N},\zeta_N]$ if $N\ge 3$ together with 
a finite \'etale morphism $\pi_{m,n}:\mathcal{A}_{g,n}\lra \mathcal{A}_{g,m}$ over $R_n$ if $m|n$ 
(cf. Theorem 6.7, p. 130 of \cite{CF} and Remark 6.2-(c), p.121 of loc.cit.). Put 
$R_{\rho,N}=R_\rho[\frac{1}{N},\zeta_N]$.   
\begin{lem}\label{compare}Let $U_N$ be the inverse image of $\mathcal{A}^{{\rm reg}}_{g,R_{\rho}}$ under the morphism 
$\pi_{1,N}/{\rm Spec}\hspace{0.5mm}R_\rho$ for $N\ge 3$. Let $X'_N=X'\times_{\mathcal{A}^{{\rm reg}}_{g,R_{\rho}}}U_N$ 
be the fiber product of $f$ and $\pi_{1,N}|_{U_N}$. 
Then it holds that 
$$\pi^\ast f_\ast\Omega^1_{X'/\mathcal{A}^{{\rm reg}}_{g,R_{\rho}}}\simeq 
f'_\ast\Omega^1_{X'_N/U_N}$$
where $\pi:=\pi_{1,N}$ and $f':X'_N\lra U_N$ is the base extension of $f$ with respect to $\pi_{1,N}|_{U_N}$.     
\end{lem}
\begin{proof}
Let us consider the following Cartesian diagram:

 \catcode`\@=11
\newdimen\cdsep
\cdsep=3em

\def\cdstrut{\vrule height .6\cdsep width 0pt depth .4\cdsep}
\def\@cdstrut{{\advance\cdsep by 2em\cdstrut}}

\def\arrow#1#2{
  \ifx d#1
    \llap{$\scriptstyle#2$}\left\downarrow\cdstrut\right.\@cdstrut\fi
  \ifx u#1
    \llap{$\scriptstyle#2$}\left\uparrow\cdstrut\right.\@cdstrut\fi
  \ifx r#1
    \mathop{\hbox to \cdsep{\rightarrowfill}}\limits^{#2}\fi
  \ifx l#1
    \mathop{\hbox to \cdsep{\leftarrowfill}}\limits^{#2}\fi
}
\catcode`\@=12

\cdsep=3em
$$
\begin{matrix}
  X'                   & \arrow{l}{\pi'={\rm pr}_{X'}}   &  X'_N     \cr
  \arrow{d}{f} &                      & \arrow{d}{f'} \cr
      \mathcal{A}^{{\rm reg}}_{g,R_{\rho}}             & \arrow{l}{\pi} & U_N     \cr
\end{matrix}
$$
By Proposition 8.10, p.175 of \cite{H}, firstly we have $\pi'^\ast\Omega^1_{\Omega^1_{X'/\mathcal{A}^{{\rm reg}}_{g,R_{\rho}}}}
\simeq \Omega^1_{X'_N/U_N}$. 
Since $\pi$ is \'etale, in particular, it is flat, hence by flat base change (see Proposition 9.3, p.255 of \cite{H}), 
we have 
$$\pi^\ast f_\ast \Omega^1_{X'/\mathcal{A}^{{\rm reg}}_{g,R_{\rho}}}=f'_\ast \pi'^\ast  \Omega^1_{X'/\mathcal{A}^{{\rm reg}}_{g,R_{\rho}}}=
f'_\ast \Omega^1_{X'_N/U_N}.
$$ 
\end{proof}
\begin{prop}\label{int}There exists an isomorphism  
$$\iota_\rho:H^0(\mathcal{A}_{g,R_{\rho}},\widetilde{\E}_\rho\otimes \widetilde{\mathcal{L}}^{\otimes k})^{{\rm TF}} 
\stackrel{\sim}{\lra} M_{\rho\otimes\det^k}({\rm Sp}_{2g}(\Z),R_{\rho})$$ as an $R_{\rho}$-module such that 
its base change to $\C$ coincides with the isomorphism in (\ref{base-change}). 
\end{prop}
\begin{proof}Let $g:X_N\lra \mathcal{A}_{g,R_{N}}$ be the universal abelian variety. 
Let $\mathcal{F}_\rho$ (resp. $\mathcal{F}'_\rho$) be the vector bundle on $\mathcal{A}_{g,R_{\rho,N}}$ 
(resp. $U_N/R_{\rho,N}$) associated to 
the Hodge bundle $g_\ast \Omega^1_{X_N/\mathcal{A}_{g,R_{N}}}$ (resp. $f'_\ast \Omega^1_{X'_N/U_N}$) and $\rho$. 
By Theorem 3, p.517 of \cite{Ghitza} $\mathcal{F}'_\rho$ is extended to 
$\mathcal{F}_{\rho}$. 
By Lemma \ref{compare} and (\ref{imp-Siegel}), since $\pi$ is finite surjective,  we have an injective map 
$$\iota'_{\rho,N}:H^0(\mathcal{A}_{g,R_\rho},\widetilde{\E}_\rho\otimes \widetilde{\mathcal{L}}^{\otimes k})^{\rm TF}=
H^0(\mathcal{A}^{{\rm reg}}_{g,R_\rho},\E_\rho\otimes \omega^{\otimes k})^{\rm TF}$$
$$\stackrel{\pi^\ast}{\hookrightarrow} H^0(U_N/R_{\rho,N},\mathcal{F}'_{\rho\otimes\det^k})^{{\rm TF}}=
H^0(\mathcal{A}_{g,R_{\rho,N}},\mathcal{F}_{\rho\otimes\det^k})^{{\rm TF}}  
$$
as a $R_\rho$ module.  Here $R_{\rho,N}$ is naturally regarded as a $R_{\rho}$-module.  Clearly the base extension of $\iota'_{\rho,N}$ to $\C$ yields 
the isomorphism in (\ref{base-change}).  

Since $N\ge 3$ we have a natural identification 
$H^0(\mathcal{A}_{g,R_{\rho,N}},\mathcal{F}_{\rho\otimes\det^k})^{{\rm TF}}=M_{\rho\otimes\det^k}(\G(N),R_{\rho,N})$ by using 
$q$-expansion principle. 
Combining it with $\iota'_N$, we have an injective morphism 
$$\iota_{\rho,N}:
H^0(\mathcal{A}_{g,R_\rho},\widetilde{\E}_\rho\otimes \widetilde{\mathcal{L}}^{\otimes k})^{\rm TF}\hookrightarrow 
M_{\rho\otimes\det^k}(\G(N),R_{\rho,N})$$
as $R_\rho$ modules.   
Applying this to $N=3$ and $N=5$ 
we see that ${\rm Im}(\iota_{\rho,3})\cap {\rm Im}(\iota_{\rho,5})$ is included in 
$$M_{\rho\otimes\det^k}(\G(3),R_{\rho,3})\cap 
M_{\rho\otimes\det^k}(\G(5),R_{\rho,5})\cap M_{\rho\otimes\det^k}({\rm Sp}_{2g}(\Z))=
M_{\rho\otimes\det^k}({\rm Sp}_{2g}(\Z),R_\rho)$$
 since $R_{\rho,3}\cap R_{\rho,5}=R_{\rho}$. 
Hence we have an injective homomorphism 
$$\iota_\rho:H^0(\mathcal{A}_{g,R_{\rho}},\widetilde{\E}_\rho\otimes \widetilde{\mathcal{L}}^{\otimes k})^{{\rm TF}} 
\lra M_{\rho\otimes\det^k}({\rm Sp}_{2g}(\Z),R_{\rho})$$ as a $R_{\rho}$-module such that 
its base change to $\C$ coincides with the isomorphism in (\ref{base-change}). What we need to prove is the 
surjectivity of $\iota_\rho$. 
Pick an element $F$ in the right hand side. We regard it as an element in $M_{\rho\otimes\det^k}(\G(N),R_{\rho,N})$ 
for some $N\ge 3$. Then by $q$-expansion principle it can be regarded as an element $H$ in 
$H^0(\mathcal{A}_{g,R_{\rho,N}},\mathcal{F}_{\rho\otimes\det^k})^{{\rm TF}}$ such that 
the finite group 
$G:=\G(1)/\G(N)$ acts trivially on $H$. By using the trace map for $\pi_{1,N}$ we see that 
$H$ belongs to $H^0(\mathcal{A}_{g,R_{\rho}},\widetilde{\E}_\rho\otimes \widetilde{\mathcal{L}}^{\otimes k})^{{\rm TF}}
\otimes_{R_{\rho}}R_{\rho,N}[\frac{1}{|G|}]$. However by $q$-expansion principle again, 
the Fourier coefficients of $H$ are all defined over $R_\rho$ since so is $F$. 
Hence $F$ belongs to $H^0(\mathcal{A}_{g,R_{\rho}},\widetilde{\E}_\rho\otimes \widetilde{\mathcal{L}}^{\otimes k})^{{\rm TF}}$. 
This completes the proof.   
\end{proof}

\begin{thm}\label{fg1}Let $\rho$ be as above. Fix a positive integer $m$. 
Let $R_\rho$ be a subring of $\Z[\frac{1}{g!}]$ such that $\rho$ is defined.
Then it holds that 
 the graded vector space $M_{\rho,m,\ast}:=\ds\bigoplus_{k\in \Z_{\ge 0}}M_{\rho\otimes\det^{km}}({\rm Sp_{2g}}(\Z),R_\rho)$ 
is finitely generated over  $M_{m,\ast}:=\ds\bigoplus_{k\in \Z_{\ge 0}}M_{\det^{km}}({\rm Sp_{2g}}(\Z),R_\rho)$. 
\end{thm}
\begin{proof}The claim follows immediately from Theorem \ref{int} and the proof of Theorem \ref{arith1}. 
\end{proof}

\begin{cor}\label{fg2}
Let $g=2$ and let $k_1\ge k_2\ge 1,\ m\ge 1$ be integers. 
The graded ring 
$$\ds\bigoplus_{k\in \Z}M_{(k_1+mk,k_2+mk)}({\rm Sp}_4(\Z),\Z)$$ is 
finitely generated over $\ds\bigoplus_{k\in \Z}M_{mk}({\rm Sp}_4(\Z),\Z)$.
\end{cor}
\begin{proof}Clearly $\rho={\rm Sym}^{k_1-k_2}{\rm St}_2\otimes\det^{k+k_2}$ is defined over $\Z$. 
Hence we can take $R_\rho=\Z$. The claim immediately follows from Theorem \ref{fg1}. 
\end{proof}

\begin{rmk}The strategy in proving Theorem \ref{fg1} may work for other congruence subgroups, as 
$\G_0(M)=
\Big\{
\left(
\begin{array}{cc}
A & B\\
C & D 
\end{array}
\right)\in {\rm Sp}_{2g}(\Z)\ \Big|\ C\equiv 0\ {\rm mod}\ M  
\Big\}$ for $M\in \Z_{>0}$. 
For example, the same claim for $\G_0(M)$ is true if we replace $R_\rho$ with $R_\rho[\frac{1}{M}]$. The results in \cite{CF} 
will be substituted into the corresponding results in \cite{L-book}.  Checking the details will be left to interested readers. 
\end{rmk}

\end{document}